\documentclass[nointlimits,11pt,oneside]{amsart}
\usepackage{amssymb,natbib,cases,enumerate}
\usepackage{color}
\usepackage{hyperref}
\usepackage[%
	a4paper,
	total={16cm,23cm},
	left=1cm, top=3cm,
	marginparsep=2pt]
{geometry}

\theoremstyle{plain}
\newtheorem{theorem}{Theorem}[section]

\newtheorem{proposition}[theorem]{Proposition}
\theoremstyle{definition}
\newtheorem{remark}[theorem]{Remark}

\numberwithin{equation}{section}

\def\vr{\varrho}

\def\M{\mathfrak M}
\def\Mpl{\mathfrak M_+}
\def\N{\mathbb N}
\def\R{\mathbb R}
\def\rn{\mathbb R^n}
\def\esssup{\operatornamewithlimits{ess\,\sup}}

\def\A{{\mathbb A}}
\def\B{{\mathbb B}}

\def\d{{\fam0 d}}
\def\L1loc{L^1_{\text{loc}}}

\begin{document}

\title[Poincar\'{e}-Sobolev inequalities on $\rn$]{Poincar\'{e}-Sobolev inequalities with rearrangement-invariant norms on the entire space}
\author{Zden\v ek Mihula}

\address{Zden\v ek Mihula, Charles University, Faculty of Mathematics and Physics, Department of Mathematical Analysis, Sokolovsk\'a~83, 186~75 Praha~8, Czech Republic}
\email{mihulaz@karlin.mff.cuni.cz}
\urladdr{0000-0001-6962-7635}

\subjclass[2000]{26D10, 46E30, 46E35, 47B38}
\keywords{optimal spaces, Poincar\'{e} inequality, Poincar\'{e}-Sobolev inequality, rearrangement-invariant spaces, Sobolev spaces}

\thanks{This research was supported by the grant P201-18-00580S of the Grant Agency of the Czech Republic, by the grant SVV-2017-260455, and by Charles University Research program No.~UNCE/SCI/023.}

\begin{abstract}
Poincar\'{e}-Sobolev-type inequalities involving rearrangement-invariant norms on the entire $\rn$ are provided. Namely, inequalities of the type $\|u-P\|_{Y(\rn)}\leq C\|\nabla^m u\|_{X(\rn)}$, where $X$ and $Y$ are either rearrangement-invariant spaces over $\rn$ or Orlicz spaces over $\rn$, $u$ is a $m-$times weakly differentiable function whose gradient is in $X$, $P$ is a polynomial of order at most $m-1$, depending on $u$, and $C$ is a constant independent of $u$, are studied. In a sense optimal rearrangement-invariant spaces or Orlicz spaces $Y$ in these inequalities when the space $X$ is fixed are found. A variety of particular examples for customary function spaces are also provided.
\end{abstract}

\date{\today}

\maketitle

\setcitestyle{numbers}
\bibliographystyle{plainnat}

%\makeatletter
   %\providecommand\@dotsep{2}
%\makeatother
%\listoftodos\relax

\section{Introduction}
Poincar\'{e}-Sobolev-type inequalities indisputably play a prominent role not only in the theory of Sobolev spaces but also in a wide range of applications in analysis of partial differential equations, calculus of variations, mathematical modeling or harmonic analysis (e.g.~\citep{MR1131493, MR2989444, MR2676347}). These types of inequalities have been exhaustively studied for decades and have been generalized in many different directions (e.g.~\citep{MR502660, MR805809, MR1683160, MR2989992, Mabook, MR0117419}). The standard Poincar\'{e}-Sobolev inequality on balls (e.g.~\citep[Corollary~1.64]{MR1461542}) can be stated as follows: if $p\in[1,n)$, $n\geq2$ (we assume this implicitly in the entire paper from now on), and $q\in[1,\frac{np}{n-p}]$, then there exists a constant $C$, depending on $p$,$q$ and on the dimension $n$ only, such that
\begin{equation}\label{intr:poincaresobolev:eq1}
\|u-\bar{u}_B\|_{L^q(B)}\leq C(n,p,q)r^{1+\frac{n}{q}-\frac{n}{p}}\|\nabla u\|_{L^p(B)},
\end{equation}
where $\bar{u}_B$ is the integral mean of $u$ over $B$ (the mean is well defined, see Section~\ref{sec:prel}), for each ball $B\subseteq\rn$ with the radius $r$ and each weakly differentiable function $u$ in $B$ whose gradient belongs to $L^p(B)$. With $q$ equal to the critical Sobolev exponent $\frac{np}{n-p}$, \eqref{intr:poincaresobolev:eq1} reads as
\begin{equation}\label{intr:poincaresobolev:eq2}
\|u-\bar{u}_B\|_{L^\frac{np}{n-p}(B)}\leq C(n,p)\|\nabla u\|_{L^p(B)}.
\end{equation}
Note that the inequality does not depend on the radius $r$ in this case. If $p=n$, then \eqref{intr:poincaresobolev:eq1} holds for each $q\in[1,\infty)$, but there is no optimal Lebesgue exponent $q$ that would render that inequality independent of $r$ as was possible with $p\in[1,n)$ and $q=\frac{np}{n-p}$. Nevertheless, the situation can be salvaged by introducing a finer scale of function spaces than the scale of Lebesgue spaces. Namely, if we substitute the critical Lebesgue space $L^n(B)$ with the smaller Lorentz space $L^{n,1}(B)$ (see Section~\ref{sec:prel} for the definition of Lorentz spaces), it can be proved (cf.~\citep{MR0146673, MR0221282}) that
\begin{equation}\label{intr:poincaresobolev:eq3}
\|u-\bar{u}_B\|_{L^\infty(B)}\leq C(n)\|\nabla u\|_{L^{n,1}(B)}.
\end{equation}

The inequalities \eqref{intr:poincaresobolev:eq2} and \eqref{intr:poincaresobolev:eq3} suggest that this type of inequality might be possibly extended to the case where balls are replaced with the entire $\rn$ provided that we find an appropriate replacement for $\bar{u}_B$. It was shown in \citep[Theorem~1.78]{MR1461542} (cf.~\citep{MR1190332}) that the inequality \eqref{intr:poincaresobolev:eq2} can be extended to the entire space. More precisely, they proved that if $p\in[1,n)$, then for each weakly differentiable function $u$ in $\rn$ whose gradient belongs to $L^p(\rn)$, there exists a (unique) $\lambda\in\R$, depending on $u$, such that
\begin{equation}\label{intr:poincaresobolev:eq4}
\|u-\lambda\|_{L^\frac{np}{n-p}(\rn)}\leq C\|\nabla u\|_{L^p(\rn)},
\end{equation}
where the constant $C$ is independent of $u$. However, the method of their proof cannot be used for extending the inequality \eqref{intr:poincaresobolev:eq3} to the entire space. Nevertheless, in \citep[Theorem~3.7]{MR3893783} (see also \citep[Theorem~3.10]{MR3893783} for a higher-order version) different techniques were used to prove that there exists a positive constant $C$ such that for each weakly differentiable function $u$ in $\rn$ whose gradient belongs to $L^{n,1}(\rn)$, there exists a $\lambda\in\R$, depending on $u$, such that
\begin{equation}\label{intr:poincaresobolev:eq5}
\|u-\lambda\|_{L^\infty(\rn)}\leq C\|\nabla u\|_{L^{n,1}(\rn)}.
\end{equation}

In this paper, not only do we present a uniform method that covers both inequalities \eqref{intr:poincaresobolev:eq4} and \eqref{intr:poincaresobolev:eq5} in a single theorem (see Theorem~\ref{thm:mainri}), but we also study these types of inequalities in much more general setting and provide, in a sense, a sharp version of such an inequality. As the discussion before \eqref{intr:poincaresobolev:eq3} suggests, the class of Lebesgue spaces is often not sufficiently fine and one needs to work with more delicate classes of function spaces, especially when limiting cases are considered. For this reason, we shall work within the rich class of rearrangement-invariant spaces, which are, loosely speaking, Banach spaces of functions whose norms depend merely on the size of functions (for precise definitions, see Section~\ref{sec:prel}). This class contains a large number of customary function spaces (e.g.~Lebesgue spaces, Lorentz spaces, Orlicz spaces, etc.) and together with rearrangement techniques is also well established in the study of Sobolev-type inequalities and related topics (e.g.~\citep{MR929981, MR3525407, MR1040269, MR1322313, MR1662313}), which makes it a natural choice for our purposes. The preceding discussion also suggests that Poincar\'{e}-Sobolev-type inequalities on the entire space are closely connected with optimal function norms in Sobolev embedding theorems (e.g.~\citep{ACPS, MR1406683, CPS, MR2254384}).

Our main results are contained in Section~\ref{sec:mainri}. We prove that, among other things, for each $m-$times weakly differentiable function $u$ in $\rn$ whose $m-$th order derivatives belong to a rearrangement-invariant space $X(\rn)$, there exists, under a natural assumption imposed on the space $X$, a polynomial $P$ of order at most $m-1$ such that
\begin{equation}\label{intr:poincaresobolev:eq6}
\|u-P\|_{X^m(\rn)}\leq C\|\nabla^m u\|_{X(\rn)},
\end{equation}
where $X^m(\rn)$ is another rearrangement-invariant space and where the constant $C$ is independent of $u$. The space $X^m(\rn)$ is, in fact, the optimal rearrangement-invariant space (i.e.~the space with the strongest rearrangement-invariant norm possible) that renders the inequality
\begin{equation*}
\|u\|_{X^m(\rn)}\leq C\|\nabla^m u\|_{X(\rn)}\quad\text{for each $u\in V_0^m X(\rn)$},
\end{equation*}
where $V_0^m X(\rn)$ is the space of all $m-$times weakly differentiable functions whose $m-$th order derivatives belong to $X(\rn)$ and whose derivatives up to order $m-1$ have ``some decay at infinity'', true (see \citep{M19} for more detail on this matter). In particular, if we plug $X=L^p$ for $p\in[1,n)$ into \eqref{intr:poincaresobolev:eq6} with $m=1$, it reads as
\begin{equation*}
\|u-\lambda\|_{L^{\frac{np}{n-p},p}(\rn)}\leq C\|\nabla u\|_{L^p(\rn)},
\end{equation*}
which is a substantial improvement over \eqref{intr:poincaresobolev:eq4} because the Lorentz norm $\|\cdot\|_{L^{\frac{np}{n-p},p}(\rn)}$ is strictly stronger than the Lebesgue norm $\|\cdot\|_{L^{\frac{np}{n-p}}(\rn)}$. If we plug $X=L^{n,1}$ into \eqref{intr:poincaresobolev:eq6} with $m=1$, we recover \eqref{intr:poincaresobolev:eq5}.

An immediate important consequence of our results (namely Theorem~\ref{thm:mainri} and Theorem~\ref{thm:examplesGLZ}) is that the weakly differentiable functions on $\rn$ whose gradient belong to the Lorentz space $L^{n, 1}(\rn)$ are not only continuous (more precisely, have continuous representatives) but also bounded. Surprisingly enough, whereas the remarkable result of Stein's that the weakly differentiable functions on $\rn$ whose gradient belong (locally) to the Lorentz space $L^{n, 1}(\rn)$ have continuous representatives (\citep{MR607898}) has been known for many years and is well-established, the boundedness of such functions is a quite recent result (recall that we assume that $n\geq2$). To the best of our knowledge, the first result in this direction was provided by Tartar in 1998. It follows from \citep[Theorem~8 and Remark~9]{MR1662313} that if $|\nabla u|\in L^{n, 1}(\rn)$, then $u$ is bounded on $\rn$ provided that $u$ has ``some decay at infinity'', that is, $\left|\left\{x\in\rn\colon|u(x)|>\lambda\right\}\right|<\infty$ for each $\lambda>0$. However, it was not until 2018 that Rabier proved in his recent paper \citep{MR3893783} that no assumptions on decay at infinity are, in fact, required. We recover his interesting result as a corollary of our general results by completely different means (see Remark~\ref{rem:boundedness}).

Although the class of rearrangement-invariant spaces is very rich and contains many customary function spaces, it is sometimes useful in applications to work within a narrower class of function spaces. A typical example of such a class is that of Orlicz spaces, which is an irreplaceable tool for analysing partial differential equations having a non-polynomial growth (e.g.~\citep{MR3328350, DONALDSON1971507, VUILLERMOT1982327}). This motivates Section~\ref{sec:mainorlicz}, where we provide a result similar to \eqref{intr:poincaresobolev:eq6} but the spaces on both sides of \eqref{intr:poincaresobolev:eq6} are from the narrower class of Orlicz spaces only.

Lastly, we provide a variety of particular examples of the inequality \eqref{intr:poincaresobolev:eq6} for customary function spaces in Section~\ref{sec:examples}. These examples include Lebesgue spaces, Lorentz spaces or Orlicz spaces.

\section{Preliminaries}\label{sec:prel}
Throughout this section, let $(R,\mu)$ be a $\sigma$-finite nonatomic measure space. We collect all the background material that will be used in this
paper here. We set
\[
	\M(R,\mu)= \{f\colon \text{$f$ is $\mu$-measurable function on $R$ with values in $[-\infty,\infty]$}\},
\]
\[
	\M_0(R,\mu)= \{f \in \M(R,\mu)\colon f \ \text{is finite}\ \mu\text{-a.e.\ on}\ R\}
\]
and
\[
\M_+(R,\mu)= \{f \in \M(R,\mu)\colon f \geq 0\}.
\]
The \textit{(unsigned) nonincreasing rearrangement} $f^* \colon  (0,\infty) \to [0, \infty ]$ of a function $f\in \M(R,\mu)$  is
defined as
\[
f^*(t)=\inf\{\lambda\in(0,\infty)\colon|\{s\in R\colon|f(s)|>\lambda\}|\leq t\},\ t\in(0,\infty).
\]
We also define the \textit{signed nonincreasing rearrangement} $f^\circ \colon  (0,\infty) \to [-\infty, \infty]$ of a function $f\in \M(R,\mu)$ as
\[
f^\circ(t)=\inf\{\lambda\in\R\colon|\{s\in R\colon f(s)>\lambda\}|\leq t\},\ t\in(0,\infty).
\]
The \textit{maximal nonincreasing rearrangement} $f^{**} \colon (0,\infty) \to [0, \infty ]$ of a function $f\in \M(R,\mu)$  is
defined as
\[
f^{**}(t)=\frac1t\int_0^ t f^{*}(s)\,\d s,\ t\in(0,\infty).
\]
If $|f|\leq |g|$ $\mu$-a.e.\ in $R$, then $f^*\leq g^*$.

If $(R,\mu)$ and $(S,\nu)$ are two (possibly different) $\sigma$-finite measure spaces, we say that functions $f\in \M(R,\mu)$ and $g\in\M(S,\nu)$ are \textit{equimeasurable} if $f^*=g^*$ on $(0,\infty)$. Functions $f$, $f^*$ and $f^\circ$ are equimeasurable.

A functional $\vr\colon  \M _+ (R,\mu) \to [0,\infty]$ is called a \textit{Banach function norm} if, for all $f$, $g$ and
$\{f_j\}_{j\in\N}$ in $\M_+(R,\mu)$, and every $\lambda \geq0$, the
following properties hold:
\begin{enumerate}[(P1)]
\item $\vr(f)=0$ if and only if $f=0$;
$\vr(\lambda f)= \lambda\vr(f)$; $\vr(f+g)\leq \vr(f)+ \vr(g)$ (the \textit{norm axiom});
\item $  f \le g$ a.e.\  implies $\vr(f)\le\vr(g)$ (the \textit{lattice axiom});
\item $  f_j \nearrow f$ a.e.\ implies
$\vr(f_j) \nearrow \vr(f)$ (the \textit{Fatou axiom});
\item $\vr(\chi_E)<\infty$ \ for every $E\subseteq R$ of finite measure (the \textit{nontriviality axiom});
\item  if $E$ is a subset of $R$ of finite measure, then $\int_{E} f\,\d\mu \le C_E
\vr(f)$ for a positive constant $C_E$, depending possibly on $E$ and $\vr$ but independent of $f$ (the \textit{local embedding in $L^1$}).
\end{enumerate}
If, in addition, $\vr$ satisfies
\begin{itemize}
\item[(P6)] $\vr(f) = \vr(g)$ whenever
$f^* = g^ *$(the \textit{rearrangement-invariance axiom}),
\end{itemize}
then we say that $\vr$ is a
\textit{rearrangement-invariant norm}.

If $\vr$ is a rearrangement-invariant norm, then the collection
\[
X=X({\vr})=\{f\in\M(R,\mu)\colon \vr(|f|)<\infty\}
\]
is called a~\textit{rearrangement-invariant space}, sometimes we shortly write just an~\textit{r.i.~space}, corresponding to the norm $\vr$. We shall write $\|f\|_{X}$ instead of $\vr(|f|)$. Note that the quantity $\|f\|_{X}$ is defined for every $f\in\M(R,\mu)$, and
\[
f\in X\quad\Leftrightarrow\quad\|f\|_X<\infty.
\]

Every rearrangement-invariant space $X$ satisfies
\begin{equation}\label{prel:vnorenidosouctu}
X\subseteq L^1+L^\infty\subseteq \M_0(R,\mu),
\end{equation}
that is, each function from $X$ can be written as a sum of a $(R,\mu)-$integrable function and a $(R,\mu)-$essentially bounded function, and is, in particular, finite $\mu-$a.e.~(see~\citep[Chapter~1, Theorem~1.4, Chapter~2, Theorem~6.6]{BS}).

With any rearrangement-invariant function norm $\vr$, there is associated another functional, $\vr'$, defined for $g \in  \M_+(R,\mu)$ as
\[
\vr'(g)=\sup\left\{\int_{R} fg\,\d\mu\colon f\in\M_+(R,\mu),\ \vr(f)\leq 1\right\}.
\]
It turns out that  $\vr'$ is also a
rearrangement-invariant norm, which is called
the~\textit{associate norm} of $\vr$. Moreover, for every rearrangement-invariant norm $\vr$ and every $f\in\Mpl(R,\mu)$, we have (see~\citep[Chapter~1, Theorem~2.9]{BS})
\[
\vr(f)=\sup\left\{\int_{R}fg\,\d\mu\colon g\in\M_+(R,\mu),\ \vr'(f)\leq 1\right\}.
\]
%By~\citep[Chapter~2, Proposition~4.2]{BS} we, in fact, have
%\[
%\vr'(g)=\sup\left\{\int_0^{\mu(R)} f^*(t)g^*(t)\,\d t\colon f\in\M(R,\mu),\ \vr(f)\leq 1\right\}
%\]
%and
%\[
%\vr(f)=\sup\left\{\int_0^{\mu(R)}f^*(t)g^*(t)\,\d t\colon g\in\M(R,\mu),\ \vr'(f)\leq 1\right\}.
%\]

If $\vr$ is a~rearrangement-invariant norm, $X=X({\vr})$ is the rearrangement-invariant space determined by $\vr$, and $\vr'$ is the associate norm of $\vr$, then the function space $X({\vr'})$ determined by $\vr'$ is called the \textit{associate space} of $X$ and is denoted by $X'$. We always have $(X')'=X$ (see~\citep[Chapter~1, Theorem~2.7]{BS}), and we shall write $X''$ instead of $(X')'$. Furthermore, the \textit{H\"older inequality}
\begin{equation}\label{prel:holder}
\int_{R}|fg|\,\d \mu\leq\|f\|_{X}\|g\|_{X'}
\end{equation}
holds for every $f,g\in \M(R,\mu)$.

%We say that a~rearrangement-invariant space $X$ is \emph{embedded into} a~rearrangement-invariant space $Y$, and we write
%\begin{equation}\label{pre:embedding}
%X\hookrightarrow Y,
%\end{equation}
%if $X\subseteq Y$ and the inclusion is continuous, that is, there exists a positive constant $C$ such that
%\begin{equation*}
%\|f\|_Y\leq C\|f\|_X\quad\text{for each }f\in X.
%\end{equation*}
%However, it turns out that \eqref{pre:embedding} holds if and only if $X\subseteq Y$ (\citep[Chapter~1, Theorem~1.8]{BS}).
%
%Another important property (see \citep[Chapter~1, Proposition~2.10]{BS}), which we shall exploit several times, is that \eqref{pre:embedding} holds if and only if
%\begin{equation}\label{pre:embeddingdual}
%Y'\hookrightarrow X'.
%\end{equation}
%Moreover, if \eqref{pre:embedding} holds, then \eqref{pre:embeddingdual} holds in fact with the same embedding constant, and vice versa.

For every rearrangement-invariant space $X$ over the measure space $(R,\mu)$, there exists a~unique rearran\-gement-invariant space $X(0,\mu(R))$ over the interval $(0,\mu(R))$ endowed with the one-dimensional Lebesgue measure such that $\|f\|_X=\|f^*\|_{X(0,\mu(R))}$ (see \citep[Chapter~2, Theorem~4.10]{BS}). This space is called the~\textit{representation space} of $X$. Throughout this paper, the representation space of a rearrangement-invariant space $X$ will be denoted by $X(0,\mu(R))$. When $R=(0,\infty)$ and $\mu$ is the Lebesgue measure, every rearrangement-invariant space $X$ over $(R,\mu)$ coincides with its representation space.

If $\vr$ is a~rearrangement-invariant norm and $X=X({\vr})$ is the
rearrangement-invariant space determined by $\vr$, we define its
\textit{fundamental function}, $\varphi_X$, by
\begin{equation*}
\varphi_X(t)=\varrho(\chi_E),\ t\in[0,\mu(R)),
\end{equation*}
where $E\subseteq R$ is such that $\mu(E)=t$. The property (P6) of rearrangement-invariant norms and the fact that $\chi_E^* = \chi_{(0, \mu(E))}$ guarantee that
the fundamental function is well defined. Moreover, one has
\begin{equation}\label{E:fundamental-relation}
	\varphi_X(t)\varphi_{X'}(t)=t
	\qquad \textup{for every}\ t\in[0,\mu(R)).
\end{equation}

Basic examples of function norms are those associated with the standard
Lebesgue spaces $L^ p$. For $p\in(0,\infty]$, we define the functional $\vr_p$
by
\[
\vr_p(f)=\|f\|_p=
\begin{cases}
\left(\int_Rf^ p\,\d\mu\right)^{\frac1p}, &\quad0<p<\infty,\\
\esssup_{R}f,&\quad p=\infty,
\end{cases}
\]
for $f \in \M_+(R,\mu)$. If $p\in[1,\infty]$, then $\vr_p$ is a rearrangement-invariant function norm.

If $0< p,q\le\infty$, we define the functional
$\vr_{p,q}$ by
\[
\vr_{p,q}(f)=\|f\|_{p,q}=
\left\|s^{\frac{1}{p}-\frac{1}{q}}f^*(s)\right\|_{q}
\]
for $f \in \M_+(R,\mu)$. The set $L^{p,q}$, defined as the collection of all $f\in\M(R,\mu)$ satisfying $\vr_{p,q}(|f|)<\infty$, is called a~\textit{Lorentz space}. If $1<p<\infty$ and $1\leq q\leq\infty$, or $p=q=1$, or $p=q=\infty$, then $\vr_{p,q}$ is equivalent to a~rearrangement-invariant function norm in the sense that there exists a~rearrangement-invariant norm $\sigma$ and a~constant $C$, $0<C<\infty$, depending on $p,q$ but independent of $f$, such that
\[
C^{-1}\sigma(f)\leq \vr_{p,q}(f)\leq C\sigma(f).
\]
As a~consequence, $L^{p,q}$ is considered to be a~rearrangement-invariant space for the above specified cases of $p,q$ (see~\citep[Chapter~4]{BS}). If either $0<p<1$ or $p=1$ and $q>1$, then $L^{p,q}$ is a~quasi-normed space. If $p=\infty$ and $q<\infty$, then $L^{p,q}=\{0\}$. For every $p\in[1,\infty]$, we have $L^{p,p}=L^{p}$. Furthermore, if $p,q,r\in(0,\infty]$ and $q\leq r$, then the inclusion $L^{p,q}\subset L^{p,r}$ holds.

If $\A=[\alpha_0,\alpha_{\infty}]\in\R^ 2$ and $t\in\R$, then we shall use the notation $\A+t=[\alpha_0+t,\alpha_{\infty}+t]$ and $t\A=[t\alpha_0, t\alpha_{\infty}]$.

Let $0<p,q\le\infty$, $\A=[\alpha_0,\alpha_{\infty}]\in\R^ 2$ and $\B=[\beta_0,\beta_{\infty}]\in\R^ 2$. Then we define the
functionals $\vr_{p,q;\A}$ and $\vr_{p,q;\A,\B}$ on $\M_+(R,\mu)$ as
\[
\vr_{p,q;\A}(f)=
\left\|t^{\frac{1}{p}-\frac{1}{q}}\ell^{\A}(t) f^*(t)\right\|_{L^ q(0,\infty)}
\]
and
\[
\vr_{p,q;\A,\B}(f)=
\left\|t^{\frac{1}{p}-\frac{1}{q}}\ell^{\A}(t)\ell\ell^{\B}(t) f^*(t)\right\|_{L^ q(0,\infty)},
\]
where
\[
\ell^{\A}(t)=
\begin{cases}
(1-\log t)^{\alpha_0}, &\quad t\in(0,1),\\
(1+\log t)^{\alpha_{\infty}}, &\quad t\in[1,\infty),
\end{cases}
\]
and
\[
\ell\ell^{\B}(t)=
\begin{cases}
(1+\log (1-\log t))^{\beta_0}, &\quad t\in(0,1),\\
(1+\log (1+\log t))^{\beta_{\infty}}, &\quad t\in[1,\infty).
\end{cases}
\]
The sets $L^{p,q;\A}$ and $L^{p,q;\A,\B}$, defined as the collections of all $f\in\M(R,\mu)$ satisfying $\vr_{p,q;\A}(|f|)<\infty$ and $\vr_{p,q;\A,\B}(|f|)<\infty$, respectively, are called \textit{Lorentz--Zygmund spaces}. The functions of the form $\ell^{\A}$, $\ell\ell^{\B}$ are called \textit{broken logarithmic functions}. It can be shown (\citep[Theorem~7.1]{OP}) that the functional $\vr_{p,q;\A}$ is equivalent to a rearrangement-invariant function norm if and only if
\begin{equation*}%\label{thm:optimalrangeexamples-condonglztobebfs}
\begin{cases}
&p=q=1,\ \alpha_0\geq0,\ \alpha_\infty\leq0\ \text{or}\\
&p\in(1,\infty)\ \text{or}\\
&p=\infty,\ q\in[1,\infty),\ \alpha_0 + \frac1{q} < 0\ \text{or}\\
&p=q=\infty,\ \alpha_0\leq0.
\end{cases}
\end{equation*}

The spaces of this type provide a common roof for many customary spaces. These include not only Lebesgue spaces and Lorentz spaces, by taking $\A=[0,0]$, but also all types of exponential and logarithmic Zygmund classes, and also the spaces discovered independently by Maz'ya (in a~somewhat implicit form involving capacitary estimates~\citep[pp.~105 and~109]{Mabook}), Hansson~\citep{Ha} and Br\'ezis--Wainger~\citep{BW}, who used it to describe the sharp target space in a limiting Sobolev embedding (the spaces can be also traced in the works of Brudnyi~\citep{B} and, in a more general setting, Cwikel and Pustylnik~\citep{CP}). One of the benefits of using broken logarithmic functions consists in the fact that the underlying measure space can be considered to have either finite or infinite measure. For the detailed study of Lorentz--Zygmund spaces we refer the reader to~\citep{OP}. In some examples presented in this paper we shall need more than two layers of logarithms. Such spaces are defined as straightforward extensions of the spaces defined above.

%We further define the spaces $L^{(p,q;\A)}$ through the functionals $\vr_{(p,q;\A)}$ given on $\M_+(R,\mu)$ by
%\[
%\vr_{(p,q;\A)}(f)=
%\left\|t^{\frac{1}{p}-\frac{1}{q}}\ell^{\A}(t) f^{**}(t)\right\|_{L^ q(0,\infty)}
%\]
%and, in an~analogous way, all the other spaces involving various levels of logarithms.

Another very important class of rearrangement-invariant spaces is the class of \emph{Orlicz spaces}. A convex, left-continuous function $A\colon [0, \infty ) \to [0, \infty ]$, neither identically zero nor infinity on $(0,\infty)$, vanishing
at $0$ is called a \emph{Young function}. Hence any Young function can be expressed in the form
\begin{equation}\label{pre:integralformofYoungfunction}
A(t) = \int _0^t a(s ) \d s \qquad \quad \text{for $t \geq 0$},
\end{equation}
 for some nondecreasing, left-continuous function $a\colon [0, \infty )\to
[0, \infty ]$. Given a Young function $A$ we define the \emph{Luxemburg norm} $\|\cdot \|_{L^A}$ as
\begin{equation*}
\|f\|_{L^A}= \inf \left\{ \lambda >0\colon\int_R A\left(
\frac{f(x)}{\lambda} \right) \d\mu(x) \leq 1 \right\},\ f \in {\Mpl(R,\mu)}.
\end{equation*}
The corresponding rearrangement-invariant space $L^A$ is called an Orlicz space. In particular,
$L^A= L^p$ if $A(t)= t^p$ when $p \in [1, \infty )$, and $L^A= L^\infty$ if $A(t)=0$ for $t\in [0, 1]$ and
$A(t) = \infty$ for $t>1$. We refer the interested reader to \citep{MR0126722, MR1113700} for more details on Orlicz spaces.

%The associate space of an Orlicz space $L^A$ is equivalent to another Orlicz space $L^{\widetilde{A}}$ where $\widetilde{A}$ is the \emph{Young conjugate function} of $A$, which is a Young function again, defined by
%\begin{equation*}
%\widetilde{A}(t) = \sup\limits_{0\leq s<\infty}\left(st-A(s)\right).
%\end{equation*}
%
We say that a Young function $A$ \emph{dominates} a Young function $B$ \emph{near zero} or \emph{near infinity} if there exist positive constants $c$ and $t_0$ such that
\begin{equation*}
B(t)\leq A(ct)\quad\text{for all $t\in[0,t_0]$ or for all $t\in[t_0,\infty)$, respectively.}
\end{equation*}
We say that two Young functions $A$ and $B$ are \emph{equivalent near zero} or \emph{near infinity} if they dominate each other near zero or near infinity, respectively.% We say that they are \emph{equivalent globally} if they are equivalent near zero and equivalent near infinity simultaneously.

If, for $F\in\Mpl(0,\infty)$, there exists $t_0>0$ such that $\int_0^{t_0}F(s)\,\d s<\infty$ or $\int_{t_0}^\infty F(s)\,\d s<\infty$, respectively, we shortly write that
\begin{equation*}
\int_0 F(s)\,\d s<\infty\text{ or }\int^\infty F(s)\,\d s<\infty\text{, respectively.}
\end{equation*}
It there does not exist such a $t_0>0$, we write
\begin{equation*}
\int_0 F(s)\,\d s=\infty\text{ or }\int^\infty F(s)\,\d s=\infty\text{, respectively.}
\end{equation*}

%If $A$ is a Young function, we define the function $h_A\colon(0,\infty)\rightarrow[0,\infty)$ by
%\begin{equation*}
%h_A(t)=\sup\limits_{0<s<\infty}\frac{A^{-1}(st)}{A^{-1}(s)},\ t>0,
%\end{equation*}
%and we set
%\begin{equation}\label{thm:optimalorliczdomain:Boydlow}
%i_A = \sup\limits_{1<t<\infty}\frac{\log t}{\log h_A(t)}
%\end{equation}
%and
%\begin{equation}\label{thm:optimalorliczdomain:Boydupp}
%I_A = \inf\limits_{0<t<1}\frac{\log t}{\log h_A(t)}.
%\end{equation}
%The quantities $i_A$ and $I_A$ are called the \emph{lower Boyd index of $A$} and the \emph{upper Boyd index of $A$}, respectively, and it can be shown that $1\leq i_A\leq I_A\leq\infty$, $i_A=\lim\limits_{t\rightarrow\infty}\frac{\log t}{\log h_A(t)}$ and $I_A=\lim\limits_{t\rightarrow0^+}\frac{\log t}{\log h_A(t)}$. We refer the interested reader to \citep{MR0126722, MR1113700} for more details on Orlicz spaces and to \citep{BS, MR0212512, MR0306887} for more details on Boyd indices.

%Given two Young functions $A$ and $B$, the function norms $\|\cdot \|_{L^A
%(0,1)}$ and $\|\cdot \|_{L^B (0,1)}$ are equivalent if and only if
%$A$ and $B$ are equivalent near infinity, in the sense that there
%exist constants $c \geq 1$ and $t_0\geq 0$ such that
%$$A(t/c) \leq B(t) \leq A(ct) \quad \hbox{for $t \geq t_0$.}$$
%\par
A common extension of  Orlicz and Lorentz spaces is provided by the
family of \emph{Orlicz-Lorentz spaces}. Given $p\in (1, \infty)$, $q\in
[1, \infty )$ and  a Young function $A$ such that
\begin{equation}\label{pre:conditiononAforOrliczLorentz}
\int^\infty \frac{A(t)}{t^{1+p}}\,\d t < \infty,
\end{equation}
we denote by $\|\cdot \|_{L(p,q,A)}$ the Orlicz-Lorentz
rearrangement-invariant function norm defined as
\begin{equation}\label{pre:orliczlorentznormdef}
\|f\|_{L(p,q,A)} = \left\|t^{-\frac1{p}} f^*(t^{\frac1{q}})\right\|_{L^A(0, \mu(R))},\quad f \in\Mpl(R,\mu).
\end{equation}
The fact that \eqref{pre:orliczlorentznormdef} actually defines  a  rearrangement-invariant function norm follows from simple variants in the proof
of \citep[Proposition 2.1]{MR2073127}. We denote by $L(p,q,A)$ the Orlicz-Lorentz space associated with the rearrangement-invariant function norm $\|\cdot
\|_{L(p,q, A)}$. Note that the class of Orlicz-Lorentz spaces includes (up to equivalent norms) the Orlicz spaces and various instances of Lorentz and Lorentz-Zygmund spaces.

In what follows we shortly denote the Lebesgue measure of a measurable set $E\subseteq\rn$ by $|E|$.

If $X$ is a rearrangement-invariant space over $\rn$ and $G\subseteq\rn$ is a measurable set, we denote by $X(G)$ the rearrangement-invariant space over $G$ corresponding to the rearrangement-invariant norm
\begin{equation*}
\rho(f) =  \|\tilde{f}\|_X,\ f\in\M_+(G),
\end{equation*}
where $\tilde{f}$ is the continuation of $f$ by $0$ to $\rn$.

Let $G\subseteq\rn$ be an open set. We shall work with \emph{Sobolev-type spaces built upon rearrangement-invariant spaces}. If $m\in\N$ and $u$ is a $m$-times weakly differentiable function on $G$, we denote by $\nabla^k u$, for $k\in\{0, 1,\dots, m\}$, the vector of all $k-$th order weak derivatives of $u$ on $G$, where $\nabla^0 u = u$. If $X(G)$ is a rearrangement-invariant space over $G$, we define the space $V^m X(G)$ by
\begin{equation*}
V^m X(G) =\{u\in\L1loc(G)\colon\text{$u$ is $m$-times weakly differentiable on $G$ and }|\nabla^m u|\in X(G)\}.%,\\
%V_0^m X(\rn) &=\{u\in V^m X(\rn)\colon|\{x\in\rn\colon |\nabla^ku(x)|>\lambda\}|<\infty\ \text{for $k\in\{0,1,\dots,m - 1\}$ and $\lambda > 0$}\} .
\end{equation*}
We stress the fact that, for a function from $V^m X(G)$, only its $m$-th order derivatives are required to be elements of $X(G)$, whereas there are no assumptions imposed on its derivatives of lower orders. The derivatives of lower orders are not required to have any regularity, we merely assume that they exist. We also write $\|\nabla^k u\|_X$ instead of $\||\nabla^k u|\|_X$ for the sake of brevity, where $|\nabla^k u|$ is the $\ell^1$-norm of the vector $\nabla^k u$.

We have that
\begin{equation}\label{prel:sobfcejsouintegrovatelnenakoulich}
V^m X(B)\subseteq L^1(B)\quad\text{for each ball in $\rn$}
\end{equation}
by \citep[Theorem~5.2.3]{Mabook}; hence the integral mean $\frac1{|B|}\int_B u(x)\d x$ of $u\in V^m X(B)$ is a well-defined finite number for each ball in $\rn$.

Throughout the paper the convention that $\frac1{\infty}=0$ and $0\cdot \infty=0$ is used without further explicit reference. If $A$ is a Young function that is equal to zero on $[0,t_0]$, then we interpret $\frac1{A(t)}$ as $\infty$ for $t\in[0,t_0]$. We write $P\lesssim Q$ when $P\leq \text{constant}\cdot Q$ where the constant is independent of appropriate quantities appearing in expressions $P$ and $Q$. Similarly, we write $P\gtrsim Q$ with the obvious meaning. We also write $P\approx Q$ when $P\lesssim Q$ and $P\gtrsim Q$ simultaneously.

%We say that a rearrangement-invariant space $Y$ over $\rn$ is the \emph{optimal target space} (within the class of rearrangement-invariant spaces) for a rearrangement-invariant space $X$ over $\rn$ in \eqref{intr:homoemb} if \eqref{intr:homoemb} is satisfied and whenever \eqref{intr:homoemb} is satisfied for another rearrangement-invariant space $Z$ over $\rn$ in place of $Y$, $Z$ is larger than $Y$, that is, $Y\hookrightarrow Z$. We say that a rearrangement-invariant space $X$ over $\rn$ is the \emph{optimal domain space} (within the class of rearrangement-invariant spaces) for a rearrangement-invariant space $Y$ over $\rn$ in \eqref{intr:homoemb} if \eqref{intr:homoemb} is satisfied and whenever \eqref{intr:homoemb} is satisfied for another rearrangement-invariant space $Z$ over $\rn$ in place of $X$, $Z$ is smaller than $X$, that is, $Z\hookrightarrow X$.

\section{Poincar\'{e}-Sobolev inequalities in rearrangement-invariant spaces}\label{sec:mainri}
The following theorem is the main result of this paper.
\begin{theorem}\label{thm:mainri}
Let $X$ be a rearrangement-invariant space over $\rn$ and $m\in\N$, $m < n$. Assume that $t^{\frac{m}{n} - 1}\chi_{(1,\infty)}(t)\in X'(0,\infty)$. Then $X^m=X^m(\sigma_m')$, where the rearrangement-invariant norm $\sigma_m$ is defined by
\begin{equation*}
\sigma_m(f)=\|t^\frac{m}{n}f^{**}(t)\|_{X'(0,\infty)},\ f\in\Mpl(\rn),
\end{equation*}
is a rearrangement-invariant space over $\rn$ and there exists a positive constant $C$, depending on $m$ and on the dimension $n$ only, such that for each $u\in V^mX(\rn)$ there exists a polynomial $P$ of order at most $m-1$, depending on $u$, that renders the inequality
\begin{equation}\label{thm:mainri:eq}
\|u-P\|_{X^m}\leq C\|\nabla^m u\|_X
\end{equation}
true. Furthermore, if $P$ and $\widetilde{P}$ are such polynomials, then $P - \widetilde{P}$ is a constant polynomial.

Moreover, if $\lim\limits_{t\rightarrow\infty}\varphi_{X^m}(t)=\infty$, then these polynomials are unique.
\end{theorem}

\begin{remark}\label{rem:mainrivysvetlenipodminky}
For example, the condition $t^{\frac{m}{n} - 1}\chi_{(1,\infty)}(t)\in X'(0,\infty)$ is satisfied for $X=L^p$ when $p\in[1,\frac{n}{m})$, and for $X=L^{\frac{n}{m},1}$.

Furthermore, the condition $t^{\frac{m}{n} - 1}\chi_{(1,\infty)}(t)\in X'(0,\infty)$ is precisely the condition that ensures the existence of the optimal rearrangement-invariant norm $\|\cdot\|_{Y}$ in the inequality
\begin{equation}\label{rem:optimalniprostor:eq}
\|u\|_{Y}\leq C\|\nabla^m u\|_X\quad\text{for each $u\in V_0^mX(\rn)$},
\end{equation}
where $V_0^mX(\rn)$ is the space of those functions $u$ from $V^mX(\rn)$ that satisfy the conditions $|\{x\in\rn\colon|\nabla^ku(x)|>\lambda\}|<\infty$ for each $\lambda>0$ and each $k\in\{0,1,\dots,m-1\}$. More precisely, it was proved in \citep[Theorem~3.1]{M19} that if $t^{\frac{m}{n} - 1}\chi_{(1,\infty)}(t)\in X'(0,\infty)$, then 
\begin{equation*}
\|u\|_{X^m}\leq C\|\nabla^m u\|_X\quad\text{for each $u\in V_0^mX(\rn)$},
\end{equation*}
where the rearrangement-invariant space $X^m$ is defined as in the theorem above, and the space $X^m$ is the optimal (i.e.~the smallest) possible rearrangement-invariant space that renders \eqref{rem:optimalniprostor:eq} true, that is, if \eqref{rem:optimalniprostor:eq} is satisfied for a rearrangement-invariant space $Y$, then $X^m\subseteq Y$. On the other hand, if $t^{\frac{m}{n} - 1}\chi_{(1,\infty)}(t)\notin X'(0,\infty)$, then there exist no rearrangement-invariant spaces $Y$ rendering \eqref{rem:optimalniprostor:eq} true.

In light of this remark, the condition $t^{\frac{m}{n} - 1}\chi_{(1,\infty)}(t)\in X'(0,\infty)$ imposed in Theorem~\ref{thm:mainri} is in fact natural and not restrictive.
\end{remark}

Before we prove Theorem~\ref{thm:mainri}, we need to establish a few results of independent interest. We start with an inequality that can be viewed as a Poincar\'{e}-Sobolev-type inequality in rearrangement-invariant spaces. We note that even though a general version of a Poincar\'{e}-Sobolev-type inequality in rearrangement-invariant spaces was established in \citep[Lemma~4.2]{MR3237035} (cf.~\citep[Lemma~4.2]{CP-ARK}), their version is not sufficient for our purposes because we need better control over the multiplicative constant appearing there. Our proof is inspired by \citep[Theorem~3.1]{MR2514054}.
\begin{theorem}\label{thm:poincare}
Let $X$ and $Y$ be rearrangement-invariant spaces over $\rn$. Assume that there exists a positive constant $C_1$ such that
\begin{equation}\label{thm:poincare:ineq}
\left\|\int_t^\infty f(s)s^{\frac{1}{n} - 1}\, \d s\right\|_{Y(0,\infty)}\leq C_1 \|f\|_{X(0,\infty)}\quad\text{for each $f\in\Mpl(0,\infty)$}.
\end{equation}
Then there exists a positive constant $C_2$, which depends on $C_1$ and on the dimension $n$ only, such that
\begin{equation*}
\|u - \bar{u}_B\|_{Y(B)}\leq C_2\|\nabla u\|_{X(B)}\quad\text{for each ball}\ B\subseteq\rn\ \text{and each}\ u\in V^1X(B),
\end{equation*}
where $\bar{u}_B$ is the integral mean of $u$ over $B$, that is, $\bar{u}_B = \frac1{|B|}\int_B u(x)\,\d x$.
\end{theorem}
\begin{proof}
Note that $\bar{u}_B$ is a well-defined finite number by \eqref{prel:sobfcejsouintegrovatelnenakoulich}. Since for any real number $\gamma$ and a.e.~$x\in B$ we have, by the H\"older inequality \eqref{prel:holder}, that
\begin{equation*}
|u(x) - \bar{u}_B|\leq |u(x) - \gamma| + \frac1{|B|}\int_B |u(y) - \gamma|\,\d y\leq |u(x) - \gamma| + \frac1{|B|}\|u -\gamma\|_{Y(B)}\|1\|_{Y'(B)},
\end{equation*}
it follows from this estimate and \eqref{E:fundamental-relation} that
\begin{equation}\label{thm:poincare:eq1}
\|u - \bar{u}_B\|_{Y(B)}\leq 2\|u-\gamma\|_{Y(B)}
\end{equation}
for each $\gamma\in\R$. Since $u^\circ$ is locally absolutely continuous (\citep[Lemma~6.6]{CEG}), we have that
\begin{equation}\label{thm:poincare:eq2}
\|u-u^\circ \left(\frac{|B|}{2}\right)\|_{Y(B)} = \|u^\circ-u^\circ \left(\frac{|B|}{2}\right)\|_{Y(0, |B|)} =\|\int_t^\frac{|B|}{2}-\frac{\d u^\circ}{\d s}(s)\,\d s\|_{Y(0, |B|)}.
\end{equation}
Hence, by virtue of \eqref{thm:poincare:eq1} and \eqref{thm:poincare:eq2}, it is sufficient to prove that
\begin{equation}\label{thm:poincare:eq3}
\|\int_t^\frac{|B|}{2}-\frac{\d u^\circ}{\d s}(s)\,\d s\|_{Y(B)}\lesssim\|\nabla u\|_{X(B)}
\end{equation}
with a constant that depends only on $C_1$ and on $n$. Since \eqref{thm:poincare:ineq} is in force, we have that
\begin{equation}\label{thm:poincare:eq4}
\begin{split}
\|\chi_{(0,\frac{|B|}{2})}(t)\int_t^\frac{|B|}{2}-\frac{\d u^\circ}{\d s}(s)\,\d s\|_{Y(0, |B|)} &= \|\chi_{(0,\frac{|B|}{2})}(t)\int_t^\frac{|B|}{2}-\frac{\d u^\circ}{\d s}(s)\chi_{(0,\frac{|B|}{2})}(s)\,\d s\|_{Y(0, |B|)}\\
&\leq\|\int_t^\infty-\frac{\d u^\circ}{\d s}(s)\chi_{(0,\frac{|B|}{2})}(s)\,\d s\|_{Y(0, \infty)}\\
&\lesssim\|-\frac{\d u^\circ}{\d t}(t)t^{1-\frac{1}{n}}\chi_{(0,\frac{|B|}{2})}(t)\|_{X(0,\infty)}.
\end{split}
\end{equation}
It is well known (e.g.~\citep[Theorem~5.4.3]{ziemerweakly}) that the isoperimetric function $h_B$ of a ball $B\subseteq\rn$ satisfies
\begin{equation*}
h_B(s)\gtrsim\min\{s, |B| - s\}^{1-\frac{1}{n}}\quad\text{for each}\ s\in(0, |B|)
\end{equation*}
with a constant that depends only on the dimension $n$. Hence
\begin{equation}\label{thm:poincare:eq5}
\|-\frac{\d u^\circ}{\d t}(t)t^{1-\frac{1}{n}}\chi_{(0,\frac{|B|}{2})}(t)\|_{X(0,\infty)}\lesssim\|-\frac{\d u^\circ}{\d t}(t)h_B(t)\|_{X(0,|B|)}\leq\|\nabla u\|_{X(B)},
\end{equation}
where the last inequality is valid thanks to \citep[Lemma~4.1]{CP-ARK}. Exploiting the rearrangement invariance and the fact that the transformation $s\mapsto (|B| - s)$ is measure preserving on $(0,|B|)$ several times (cf.~\citep[Chapter~2, Proposition~7.2]{BS}) together with \eqref{thm:poincare:ineq}, we estimate in a similar way to \eqref{thm:poincare:eq4} and \eqref{thm:poincare:eq5} that
\begin{align*}
\|\chi_{(\frac{|B|}{2}, |B|)}(t)\int_t^\frac{|B|}{2}-\frac{\d u^\circ}{\d s}(s)\,\d s\|_{Y(0, |B|)} &= \|\chi_{(\frac{|B|}{2}, |B|)}(t)\int_{|B| - t}^\frac{|B|}{2}-\frac{\d u^\circ}{\d s}(|B| - s)\,\d s\|_{Y(0, |B|)}\\
&=\|\chi_{(\frac{|B|}{2}, |B|)}(|B| - t)\int_t^\frac{|B|}{2}-\frac{\d u^\circ}{\d s}(|B| - s)\,\d s\|_{Y(0, |B|)}\\
&=\|\chi_{(0, \frac{|B|}{2})}(t)\int_t^\frac{|B|}{2}-\frac{\d u^\circ}{\d s}(|B| - s)\,\d s\|_{Y(0, |B|)}\\
&\lesssim\|-\frac{\d u^\circ}{\d t}(|B| - t)t^{1-\frac{1}{n}}\chi_{(0,\frac{|B|}{2})}(t)\|_{X(0,|B|)}\\
&=\|-\frac{\d u^\circ}{\d t}(t)(|B| - t)^{1-\frac{1}{n}}\chi_{(0,\frac{|B|}{2})}(|B| - t)\|_{X(0,|B|)}\\
&=\|-\frac{\d u^\circ}{\d t}(t)(|B| - t)^{1-\frac{1}{n}}\chi_{(\frac{|B|}{2}, |B|)}(t)\|_{X(0,|B|)}\\
&\lesssim\|-\frac{\d u^\circ}{\d t}(t)h_B(t)\|_{X(0,|B|)}\\
&\leq\|\nabla u\|_{X(B)}.
\end{align*}
Hence
\begin{equation}\label{thm:poincare:eq6}
\|\chi_{(\frac{|B|}{2}, |B|)}(t)\int_t^\frac{|B|}{2}-\frac{\d u^\circ}{\d s}(s)\,\d s\|_{Y(0, |B|)}\lesssim\|\nabla u\|_{X(B)}.
\end{equation}

Finally, the combination of \eqref{thm:poincare:eq4}, \eqref{thm:poincare:eq5} and \eqref{thm:poincare:eq6} establishes \eqref{thm:poincare:eq3}.
\end{proof}

The following proposition tells us that a sequence of integral means of a function from $V^1 X(\rn)$ over increasing balls is bounded under an assumption on $X$. This enables us to find the desired polynomials in \eqref{intr:poincaresobolev:eq6} later.
\begin{proposition}\label{prop:existenceprumeru}
Assume that $X$ is a rearrangement-invariant space over $\rn$ that satisfies $t^{\frac{1}{n} - 1}\chi_{(1,\infty)}(t)\in X'(0,\infty)$. Let $u\in V^1 X(\rn)$ and set
\begin{equation*}
\lambda_k = \frac1{|B_k|}\int_{B_k} u(x)\,\d x,\ k\in\N,
\end{equation*}
where $B_k$ are balls in $\rn$ such that $B_k\subseteq B_{k+1}$ for each $k\in\N$. Then the sequence $\{\lambda_k\}_{k=1}^\infty$ is bounded. In particular, it has a convergent subsequence.
\end{proposition}
\begin{proof}
Since $t^{\frac{1}{n} - 1}\chi_{(1,\infty)}(t)\in X'(0,\infty)$, there exists a rearrangement-invariant space $Y$ over $\rn$ such that $I_1\colon X\rightarrow Y$ by \citep[Theorem~6.3]{EMMP}, where the operator $I_1$ (the Riesz potential of order one) is defined, for a  function $f\in\L1loc(\rn)$, as $I_1f(x)=\int_{\rn}\frac{f(y)}{|x-y|^{n-1}}\,\d y,\ x\in\rn$. In particular, it follows that
\begin{equation}\label{prop:existenceprumeru:eq1}
I_1(|\nabla u|)(x) < \infty\quad\text{and}\quad |u(x)|<\infty\quad\text{for a.e.~$x\in\rn$}
\end{equation}
by \eqref{prel:vnorenidosouctu} and $V^1X(\rn)\subseteq\L1loc(\rn)$.

Note that we have that $u\in L^1(B_k)$ and $|\nabla u|\in L^1(B_k)$ for each $k\in\N$ thanks to $V^1X(\rn)\subseteq\L1loc(\rn)$ and $X\subseteq\L1loc(\rn)$. This allows us to exploit a standard potential estimate (e.g.~\cite[Lemma~1.50]{MR1461542}) to obtain that
\begin{equation}\label{prop:existenceprumeru:eq2}
|u(x)-\lambda_k|\lesssim I_1\left(|\nabla u|\chi_{B_k}\right)(x)\leq I_1\left(|\nabla u|\right)(x)\quad\text{for each $k\in\N$ and a.e.~$x\in B_k$},
\end{equation}
where the multiplicative constant depends on the dimension $n$ only. Coupling \eqref{prop:existenceprumeru:eq1} and \eqref{prop:existenceprumeru:eq2} together, we obtain that there exists a point $x_0\in B_1\subseteq B_k$ such that
\begin{equation*}
|u(x_0)|<\infty\quad\text{and}\quad|u(x_0)-\lambda_k|\lesssim I_1\left(|\nabla u|\right)(x_0) < \infty\quad\text{for each $k\in\N$}
\end{equation*}
whence the desired boundedness immediately follows from
\begin{equation*}
|\lambda_k|\lesssim I_1\left(|\nabla u|\right)(x_0) + |u(x_0)|\quad\text{for each $k\in\N$}.
\end{equation*}
\end{proof}

Now, we are finally prepared to prove our main result.
\begin{proof}[Proof of Theorem~\ref{thm:mainri}]
We just remark that $X^m$ is indeed a rearrangement-invariant space by \citep[Theorem~3.1]{M19}.

We start by proving the uniqueness part. Assume that for $u\in V^mX(\rn)$ there exist polynomials $P$ and $\widetilde{P}$ (of order at most $m-1$) such that
\begin{equation*}
\|u-P\|_{X^m}\leq C\|\nabla^m u\|_X
\end{equation*}
and also
\begin{equation*}
\|u-\widetilde{P}\|_{X^m}\leq C\|\nabla^m u\|_X.
\end{equation*}
Hence $P-\widetilde{P}\in X^m$. Since $X^m\subseteq (L^1+L^\infty)$ by \eqref{prel:vnorenidosouctu}, it follows that $P-\widetilde{P}$ is a constant polynomial.

Moreover, we have that
\begin{equation*}
c\varphi_{X^m}(k) = \|(P-\widetilde{P})\chi_{E_k}\|_{X^m}\leq\|u-P\|_{X^m} + \|u-\widetilde{P}\|_{X^m}\lesssim\|\nabla^m u\|_X < \infty \quad\text{for each $k\in\N$},
\end{equation*}
where $c=|P-\widetilde{P}|$ and $E_k\subseteq\rn$ are such that $|E_k|=k$. Hence $c=0$ if $\lim\limits_{t\rightarrow\infty}\varphi_{X^m}(t)=\infty$.

%Assume that $P\neq\widetilde{P}$. Since $P-\widetilde{P}$ is a nonzero polynomial, we have $\lim\limits_{|x|\rightarrow\infty}|P(x)-\widetilde{P}(x)|>0$. Therefore, there exist a positive number $A$ and a sequence of points $y_k\in\rn$ such that
%\begin{equation*}
%|P(x)-\widetilde{P}(x)|\geq A>0\quad\text{for each $x\in B(y_k,2^{k})$ and $k\in\N$}
%\end{equation*}
%whence
%\begin{equation}\label{thm:mainri:eq3}
%A\varphi_{X^m}(|B(y_k,2^{k})|)\leq \|(P-\widetilde{P})\chi_{B(y_k,2^{k})}\|_{X^m}\leq\|u-P\|_{X^m} + \|u-\widetilde{P}\|_{X^m}\lesssim\|\nabla^m u\|_X < \infty.
%\end{equation}

%As $\quad\sum\limits_{k=1}^\infty\frac1{\varphi_{X^m}(2^{kn})}<\infty$, we have $\lim\limits_{k\rightarrow\infty}\varphi_{X^m}(|B(y_k,2^{k})|) = \infty$. However, letting $k\rightarrow\infty$ in \eqref{thm:mainri:eq3}, since $A>0$, we obtain a contradiction; hence $P=\widetilde{P}$.\textcolor{red}{!!!---}\komentar{Obecne zatim jednoznacnost az na aditivni konstantu.}

We now proceed to prove the existence of such polynomials by induction on $m$. Assume that $m=1$.

Let $u\in V^1X(\rn)$ and define
\begin{equation*}
\lambda_k = \frac1{|B_k|}\int_{B_k}u(x)\,\d x,\ k\in\N,
\end{equation*}
where $B_k$ is the ball in $\rn$ centered at the origin having its radius equal to $k$. By Proposition~\ref{prop:existenceprumeru} we may assume without loss of generality that there exists a $\lambda\in\R$ such that
\begin{equation}\label{thm:mainri:eq1}
\lim\limits_{k\rightarrow\infty}\lambda_k = \lambda.
\end{equation}
%Indeed, Theorem~\ref{thm:reductionprinciple} coupled with Theorem~\ref{thm:optimalrange} tells us that \eqref{thm:eq:equivalentformsofonedimensionalinequality:reduction} with $m=1$ and $Y(0,\infty) = X^1(0,\infty)$ is true; therefore we can use Proposition~\ref{thm:poincare} to obtain that
%\begin{align*}
%|\lambda_{k+1} - \lambda_{k}|&\leq\frac1{|B_k|}\int_{B_k}|u(x) - \lambda_{k+1}|\,\d x\lesssim\frac1{|B_{k+1}|}\int_{B_{k+1}}|u(x) - \lambda_{k+1}|\,\d x\\
%&\leq\frac1{|B_{k+1}|}\|u(x) - \lambda_{k+1}\|_{X^1(B_{k+1})}\|1\|_{(X^1)'(B_{k+1})}\\
%&= \frac1{|B_{k+1}|}\|u(x) - \lambda_{k+1}\|_{X^1(B_{k+1})}\frac{|B_{k+1}|}{\varphi_{X^1}(|B_{k+1}|)}\\
%&\lesssim\frac1{\varphi_{X^1}(|B_{k+1}|)}\|\nabla u\|_{X(B_{k+1})}\leq\frac1{\varphi_{X^1}(|B_{k+1}|)}\|\nabla u\|_X,
%\end{align*}
%where we used \eqref{E:fundamental-relation} in the third inequality, with a constant independent of $u$ and $k$. Since $\varphi_{X^1}(|B_{k+1}|)\approx\varphi_{X^1}(2^{kn})$ with constants depending on the dimension $n$ only, it follows from \eqref{thm:mainri:sufcond} and the estimate above that $\{\lambda_k\}_{k = 1}^\infty$ is a Cauchy sequence, which proves \eqref{thm:mainri:eq1}.

Since the inequality \eqref{thm:poincare:ineq} with $Y=X^1$ holds with $C_1=1$ by virtue of \citep[Theorem~3.3]{M19} and the definition of $\sigma_m$, we have that
\begin{equation}\label{thm:mainri:eq2}
\|u-\lambda_k\|_{X^1(B_k)}\lesssim\|\nabla u\|_{X(B_k)}\leq\|\nabla u\|_X,
\end{equation}
where the multiplicative constant depends only on $m$ and on $n$, by Theorem~\ref{thm:poincare}. Furthermore, we have that
\begin{equation*}
\lim\limits_{k\rightarrow\infty}(u(x)-\lambda_k)\chi_{B_k}(x) = u(x) - \lambda\quad\text{for a.e.}\ x\in\rn
\end{equation*}
by \eqref{thm:mainri:eq1}. Hence
\begin{equation*}
\|u-\lambda\|_{X^1}\leq\liminf\limits_{k\rightarrow\infty}\|u-\lambda_k\|_{X^1(B_k)}\lesssim\|\nabla u\|_X
\end{equation*}
by Fatou's lemma \citep[Chapter~1, Theorem~1.7]{BS} and \eqref{thm:mainri:eq2}. This completes the proof for $m=1$.

Finally, for the inductive step, assume that $1<m<n$ and let $u\in V^mX(\rn)$. Then $\frac{\partial^{|\alpha|}u}{\partial^\alpha x}\in V^1X(\rn)$ for each multi-index $\alpha\in\N_0^n$ such that $|\alpha|=\alpha_1+\cdots+\alpha_n=m-1$. By the induction hypothesis for $m=1$ (clearly $t^{\frac{1}{n} - 1}\chi_{(1,\infty)}(t)\in X'(0,\infty)$ if $t^{\frac{m}{n} - 1}\chi_{(1,\infty)}(t)\in X'(0,\infty)$), there exist $\lambda_\alpha\in\R$ such that
\begin{equation}\label{thm:mainri:eq4}
\|\frac{\partial^{|\alpha|}u}{\partial^\alpha x}-\lambda_\alpha\|_{X^1}\lesssim\|\nabla^m u\|_X\quad\text{for each $\alpha\in\N_0^n$, $|\alpha|=m-1$}.
\end{equation}
Set
\begin{equation*}
v(x)=u(x)-\sum_{\substack{\alpha\in\N_0^n\\|\alpha|=m-1}}\lambda_\alpha x^\alpha,\ \text{a.e.~$x\in\rn$}.
\end{equation*}
Since
\begin{equation}\label{thm:mainri:eq5}
\frac{\partial^{|\alpha|}v}{\partial^\alpha x}=\frac{\partial^{|\alpha|}u}{\partial^\alpha x}-\lambda_\alpha\quad\text{for each $\alpha\in\N_0^n$, $|\alpha|=m-1$},
\end{equation}
we have $v\in V^{m-1}X^1(\rn)$ by \eqref{thm:mainri:eq4}. By the induction hypothesis again (this time for $m-1$), there exists a polynomial $Q$ of order at most $m-2$ such that
\begin{align*}
\|u-(Q+\sum_{\substack{\alpha\in\N_0^n\\|\alpha|=m-1}}\lambda_\alpha x^\alpha)\|_{X^m}&=\|v-Q\|_{X^m}\approx\|v-Q\|_{\left(X^1\right)^{m-1}}\lesssim\|\nabla^{m-1}v\|_{X^1}\\
&\approx\max\left\{\|\frac{\partial^{|\alpha|}v}{\partial^\alpha x}\|_{X^1}\colon\alpha\in N_0^n,|\alpha|=m-1\right\}\\
&=\max\left\{\|\frac{\partial^{|\alpha|}u}{\partial^\alpha x}-\lambda_\alpha\|_{X^1}\colon\alpha\in N_0^n,|\alpha|=m-1\right\}\\
&\lesssim\|\nabla^m u\|_X,
\end{align*}
where the first equivalence follows from \citep[Theorem~3.2]{M19}, the second equality is \eqref{thm:mainri:eq5} and the last inequality is \eqref{thm:mainri:eq4}. We complete the proof by observing that $Q+\sum\limits_{\substack{\alpha\in\N_0^n\\|\alpha|=m-1}}\lambda_\alpha x^\alpha$ is a polynomial of order at most $m-1$.
\end{proof}

\section{Poincar\'{e}-Sobolev inequalities in Orlicz spaces}\label{sec:mainorlicz}
In this section we provide a result concerning Orlicz spaces that is very similar to Theorem~\ref{thm:mainri}. The key difference here is that in the following theorem the resulting function norm is an Orlicz norm.

Let $m < n$ and let $A$ be a Young function satisfying
\begin{equation}\label{thm:mainorlicz:conA}
\int_0 \left(\frac{s}{A(s)}\right)^\frac{m}{n-m}\,\d s < \infty.
\end{equation}
We set
\begin{equation*}
H^\infty =\lim\limits_{t\rightarrow\infty}H_m(t)
\end{equation*}
where $H_m$ is defined by
\begin{equation*}
H_m(t) = \left(\int_0^t \left(\frac{s}{A(s)}\right)^\frac{m}{n-m}\,\d s\right)^\frac{n-m}{n},\ t\geq0.
\end{equation*}
Note that $H^\infty = \infty$ if and only if
\begin{equation*}
\int^\infty \left(\frac{s}{A(s)}\right)^\frac{m}{n-m}\,\d s = \infty.
\end{equation*}
Observe that $H_m$ is an increasing function onto $[0, H_\infty)$ (see \citep[Remark~3.3.1]{VejtekPhD}); hence its inverse function $H_m^{-1}$ is well defined on $[0, H_\infty)$. Lastly, we define
\begin{equation}\label{thm:mainorlicz:defDm}
D_m(t)=\begin{cases}
\left(t\frac{A(H_m^{-1}(t))}{H_m^{-1}(t)}\right)^\frac{n}{n-m},\quad&0\leq t < H^\infty,\\
\infty,\quad&H^\infty\leq t <\infty.
\end{cases}
\end{equation}
\begin{theorem}\label{thm:mainorlicz}
Let $m < n$ and let $A$ be a Young function satisfying \eqref{thm:mainorlicz:conA}. Define the function $A_m$ by
\begin{equation*}
A_m(t) = \int_0^t\frac{D_m(s)}{s}\,\d s,\ t\geq0,
\end{equation*}
where the function $D_m$ is defined by \eqref{thm:mainorlicz:defDm}. Then $A_m$ is a Young function and there exists a positive constant $C$ such that for each $u\in V^mL^A(\rn)$ there exists a unique polynomial $P$ of order at most $m-1$, depending on $u$, that renders the inequality
\begin{equation*}
\|u-P\|_{L^{A_m}}\leq C\|\nabla^m u\|_{L^A}
\end{equation*}
true.
\end{theorem}
\begin{proof}
Set $X=L^A$. Then, combining \citep[Theorem~6.1]{M19} with \citep[Theorem~3.1]{M19}, $A_m$ is a Young function, $t^{\frac{m}{n} - 1}\chi_{(1,\infty)}(t)\in X'(0,\infty)$ and 
\begin{equation*}
\|v\|_{L^A}\lesssim\|v\|_{X^m}\quad\text{for each $v\in\M(\rn)$},
\end{equation*}
where $X^m$ is defined as in Theorem~\ref{thm:mainri}. Our claim then follows from Theorem~\ref{thm:mainri}.
\end{proof}

\begin{remark}\label{rem:mainrivsorlicz}
If a Young function $A$ satisfies \eqref{thm:mainorlicz:conA}, then $t^{\frac{m}{n} - 1}\chi_{(1,\infty)}(t)\in X'(0,\infty)$, where $X=L^A$. Therefore, we may use both Theorem~\ref{thm:mainri} and Theorem~\ref{thm:mainorlicz}. It follows (see the proof above) that
\begin{equation*}
\|u-P\|_{L^{A_m}}\lesssim \|u-P\|_{X^m}\lesssim\|\nabla^m u\|_{L^A}.
\end{equation*}
Hence the inequality that we obtain from Theorem~\ref{thm:mainorlicz} is weaker that the inequality that we obtain from Theorem~\ref{thm:mainri}. However, a possible advantage of Theorem~\ref{thm:mainorlicz} over Theorem~\ref{thm:mainri}, which may be useful in some applications, is that the resulting space $L^{A_m}$ is again an Orlicz space, not a general rearrangement-invariant space.

Furthermore (cf.~Remark~\ref{rem:mainrivysvetlenipodminky}), the convergence of the integral \eqref{thm:mainorlicz:conA} is precisely the condition that ensures the existence of the optimal Orlicz norm $\|\cdot\|_{L^B}$ in the inequality
\begin{equation}\label{rem:optimalniorlicz:eq}
\|u\|_{L^B}\leq C\|\nabla^m u\|_{L^A}\quad\text{for each $u\in V_0^mL^A(\rn)$}.
\end{equation}
More precisely, it was proved in \citep[Theorem~6.1]{M19} that if a Young function $A$ satisfies \eqref{thm:mainorlicz:conA}, then 
\begin{equation*}
\|u\|_{L^{A_m}}\leq C\|\nabla^m u\|_{L^A}\quad\text{for each $u\in V_0^mL^A(\rn)$},
\end{equation*}
where the Orlicz space $L^{A_m}$ is defined as in the theorem above, and the space $L^{A_m}$ is the optimal (i.e.~the smallest) possible Orlicz space that renders \eqref{rem:optimalniorlicz:eq} true, that is, if \eqref{rem:optimalniorlicz:eq} is satisfied for an Orlicz space $L^B$, then $L^{A_m}\subseteq L^B$. On the other hand, if $A$ does not satisfy \eqref{thm:mainorlicz:conA}, then there exist no Orlicz spaces $L^B$ (in fact, even no rearrangement-invariant spaces at all, cf.~\citep[Theorem~5.2]{M19}) rendering \eqref{rem:optimalniorlicz:eq} true.
\end{remark}

\section{Examples}\label{sec:examples}
In this section we provide important particular examples of possible applications of Theorem~\ref{thm:mainri} and Theorem~\ref{thm:mainorlicz}. These examples include a large number of customary function spaces (e.g.~Lebesgue spaces, Lorentz spaces, Orlicz spaces, Zygmund classes, etc.).

The following theorem is a combination of Theorem~\ref{thm:mainri} and \citep[Theorem~5.1]{M19}
\begin{theorem}\label{thm:examplesGLZ}
Let $m\in\N$, $m < n$, and let $p\in[1,\frac{n}{m}]$, $q\in[1,\infty]$ and $\A=[\alpha_0,\alpha_{\infty}]\in\R^2$. Then there exists a constant $C$ such that for each $u\in V^m L^{p, q; \A}(\rn)$ there exists a polynomial $P$, which depends on $u$, of order at most $m-1$ such that
\begin{equation*}
\|u-P\|_{Y}\leq C\|\nabla^m u\|_{L^{p, q; \A}}
\end{equation*}
where
\begin{equation*}
Y =\begin{cases}L^{\frac{np}{n-mp}, q;\A},\quad&p=q=1,\ \alpha_0\geq0,\ \alpha_\infty\leq0\ \text{or}\\
&p\in(1,\frac{n}{m}),\\
L^{\infty, q;\A - 1},\quad&p=\frac{n}{m},\ \alpha_0 < \frac1{q'},\ \alpha_\infty>\frac1{q'},\\
L^{\infty, 1;[-1,\alpha_\infty - 1], [-1,0], [-1,0]},\quad&p=\frac{n}{m},\ q = 1,\ \alpha_0 = 0,\ \alpha_\infty > 0,\\
Y_1,\quad&p=\frac{n}{m},\ q = 1,\ \alpha_0 < 0 ,\ \alpha_\infty = 0,\\
L^\infty,\quad&p=\frac{n}{m},\ q = 1,\ \alpha_0 \geq 0 ,\ \alpha_\infty = 0,\\
Y_2,\quad&p=\frac{n}{m},\ q\in[1,\infty),\ \alpha_0 > \frac1{q'} ,\ \alpha_\infty > \frac1{q'},\\
L^{\infty, q;[-\frac1{q},\alpha_\infty - 1],[-1,0]},\quad&p=\frac{n}{m},\ q\in(1,\infty],\ \alpha_0 = \frac1{q'},\ \alpha_\infty >\frac1{q'},\\
L^{\infty, \infty;[0,\alpha_\infty - 1]},\quad&p=\frac{n}{m},\ q = \infty,\ \alpha_0 > 1,\ \alpha_\infty > 1,\end{cases}
\end{equation*}
where
\begin{align*}
\|f\|_{Y_1} &=\|t^{-1}\ell^{\alpha_0 - 1}(t)f^*(t)\|_{L^1(0,1)},\\
\|f\|_{Y_2} &= \|f\|_{L^\infty} + \|t^{-\frac1{q}}\ell^{\alpha_\infty - 1}(t)f^*(t)\|_{L^q(1,\infty)}.
\end{align*}
Moreover, these polynomials $P$ are unique in all the cases above except for the case $p=\frac{n}{m}$, $q=1$, and $\alpha_\infty = 0$.
\end{theorem}

\begin{remark}\label{rem:boundedness}
In particular, it follows from the theorem above that if $|\nabla u|\in L^{n,1}(\rn)$, then $u$ is bounded on $\rn$ (more precisely, its continuous representative is). If $m\geq2$ and $|\nabla^m u|\in L^{\frac{n}{m},1}(\rn)$, then $u$ need not be bounded on the entire space (consider, for example, $u(x)=x$), but it differs from a polynomial of order at most $m-1$ by a bounded function on $\rn$. However, if $|\nabla^m u|\in L^{\frac{n}{m},1}(\rn)$ ($m\geq2$) and $u$ has ``some decay at infinity'', that is, $\left|\left\{x\in\rn\colon|u(x)|>\lambda\right\}\right|<\infty$ for each $\lambda>0$, it follows again that $u$ is not only continuous but also bounded on the entire space.
\end{remark}

Important examples of Orlicz spaces, frequently appearing in applications, are so-called Zygmund classes $L^p\left(\log L\right)^\alpha$ (\citep{MR0107776}) and their different variations. It is worth noting that even though the more general variations of Zygmund classes contained in the following theorem are special instances of Lorentz-Zygmund spaces when $p_0=p_\infty$ (\citep[Chapter~8]{OP}), which were considered in the preceding theorem, and the Orlicz spaces $L^B$ from the following theorem are in general ``worse'' (i.e.~$\|\cdot\|_{L^B}$ is a weaker norm than $\|\cdot\|_{Y}$) than the rearrangement-invariant spaces $Y$ from the preceding theorem corresponding to the same $X=L^A$ (see Remark~\ref{rem:mainrivsorlicz}), the importance of the following theorem stems from the fact that the spaces $L^B$ themselves are from the narrower class of Orlicz spaces. The following theorem is a combination of Theorem~\ref{thm:mainorlicz} and \citep[Theorem~6.3]{M19}
\begin{theorem}%\label{thm:examplesOrlicz}
Let $m\in\N$, $m < n$, and let $p_0\in[1,\frac{n}{m}]$, $p_\infty\in[1,\infty)$ and $\alpha_0,\,\alpha_\infty\in\R$. Let $A(t)$ be a Young function that is equivalent to
\begin{equation*}
\begin{cases}
t^{p_0}\ell^{\alpha_0}(t)\quad&\text{near zero},\\
t^{p_\infty}\ell^{\alpha_\infty}(t)\quad&\text{near infinity}.
\end{cases}
\end{equation*}

Then there exists a constant $C$ such that for each $u\in V^mL^A(\rn)$ there exists a unique polynomial $P$, which depends on $u$, of order at most $m-1$ such that
\begin{equation*}
\|u-P\|_{L^B}\leq C\|\nabla^m u\|_{L^A}
\end{equation*}
where the Young function $B(t)$ is equivalent to
\begin{equation*}
\begin{cases}
t^{\frac{np_0}{n-mp_0}}\ell^{\frac{n\alpha_0}{n-mp_0}}(t),\quad&\text{$p_0=1$, $\alpha_0\leq0$ or}\\
																															&p_0\in(1,\frac{n}{m}),\\
e^{-t^{\frac{n}{n-(1+\alpha_0)m}}},\quad&\text{$p_0=\frac{n}{m}$, $\alpha_0>\frac{n-m}{m}$},
\end{cases}
\end{equation*}
near zero and to
\begin{equation*}
\begin{cases}
t^{\frac{np_\infty}{n-mp_\infty}}\ell^{\frac{n\alpha_\infty}{n-mp_\infty}}(t),\quad&\text{$p_\infty=1$, $\alpha_\infty\geq0$ or}\\
																																							&p_\infty\in(1,\frac{n}{m}),\\
e^{t^{\frac{n}{n-(1+\alpha_\infty)m}}},\quad&\text{$p_\infty=\frac{n}{m}$, $\alpha_\infty < \frac{n-m}{m}$},\\
e^{e^{t^{\frac{n}{n - m}}}},\quad&\text{$p_\infty=\frac{n}{m}$, $\alpha_\infty = \frac{n-m}{m}$},\\
\infty,\quad&\text{$p_\infty=\frac{n}{m}$, $\alpha_\infty>\frac{n-m}{m}$ or}\\
\quad&\text{$p_\infty\in(\frac{n}{m},\infty)$},
\end{cases}
\end{equation*}
near infinity.
\end{theorem}

In our last example we couple Theorem~\ref{thm:mainri} with \citep[Theorem~5.2]{M19} to obtain a specific version of \eqref{thm:mainri:eq} when $X$ is an Orlicz space $L^A$. It turns out that we need to distinguish whether the integral
\begin{equation}\label{thm:optimalrangerforOrlicz:connek}
\int^\infty \left(\frac{s}{A(s)}\right)^\frac{m}{n-m}\,\d s
\end{equation}
converges or diverges. Assume that $m<n$ and that $A$ is a Young function satisfying \eqref{thm:mainorlicz:conA}. Let $a$ be the left-continuous derivative of $A$, that is, $a$ and $A$ are related as in \eqref{pre:integralformofYoungfunction}. We define a function $E_m$ by
\begin{equation}\label{thm:optimalrangerforOrlicz:defEm}
E_m(t) = \int _0^t e_m(s)\,\d s,\ t\geq0,
\end{equation}
where $e_m$ is the nondecreasing, left-continuous function in $[0, \infty)$ satisfying
\begin{equation*}
e_m^{-1}(t) =  \left(\int _{a^{-1}(t)}^{\infty}\left(\int _0^s \left(\frac{1}{a(\tau)}\right)^{\frac{m}{n-m}}\,\d\tau\right)^{-\frac{n}{m}}\frac1{a(s)^{\frac{n}{n-m}}}\,\d s\right)^{\frac{m}{m-n}}\quad\text{for $t\geq0$}.
\end{equation*}
Then $E_m$ is a finite-valued Young function satisfying \eqref{pre:conditiononAforOrliczLorentz} with $p = \frac{n}{m}$ (see~\citep[Proposition~2.2]{MR2073127}).
\begin{theorem}
Let $m\in\N$, $m < n$, and let $A$ be a Young function satisfying \eqref{thm:mainorlicz:conA}. Then there exists a constant $C$ such that for each $u\in V^m L^A(\rn)$ there exists a polynomial $P$, which depends on $u$, of order at most $m-1$ such that
\begin{equation*}
\|u-P\|_{Y}\leq C\|\nabla^m u\|_{L^A}
\end{equation*}
where
\begin{equation*}
Y=\begin{cases}
L(\frac{n}{m}, 1, E_m)\quad&\text{if the integral \eqref{thm:optimalrangerforOrlicz:connek} diverges},\\
L(\frac{n}{m}, 1, E_m)\cap L^\infty\quad&\text{if the integral \eqref{thm:optimalrangerforOrlicz:connek} converges}.
\end{cases}
\end{equation*}
The Young function $E_m$ is defined by \eqref{thm:optimalrangerforOrlicz:defEm}.
\end{theorem}

\bibliography{\jobname}

\begin{thebibliography}{45}
\providecommand{\natexlab}[1]{#1}
\providecommand{\url}[1]{\texttt{#1}}
\expandafter\ifx\csname urlstyle\endcsname\relax
  \providecommand{\doi}[1]{doi: #1}\else
  \providecommand{\doi}{doi: \begingroup \urlstyle{rm}\Url}\fi

\bibitem[Alberico et~al.(2018)Alberico, Cianchi, Pick, and
  Slav\'{i}kov\'{a}]{ACPS}
A.~Alberico, A.~Cianchi, L.~Pick, and L.~Slav\'{i}kov\'{a}.
\newblock Sharp {S}obolev type embeddings on the entire {E}uclidean space.
\newblock \emph{Commun. Pure Appl. Anal.}, 17\penalty0 (5):\penalty0
  2011--2037, 2018.
\newblock ISSN 1534-0392.
\newblock \doi{10.3934/cpaa.2018096}.
\newblock URL \url{https://doi.org/10.3934/cpaa.2018096}.

\bibitem[Alves et~al.(2014)Alves, Figueiredo, and Santos]{MR3328350}
C.O. Alves, G.M. Figueiredo, and J.A. Santos.
\newblock Strauss and {L}ions type results for a class of {O}rlicz-{S}obolev
  spaces and applications.
\newblock \emph{Topol. Methods Nonlinear Anal.}, 44\penalty0 (2):\penalty0
  435--456, 2014.
\newblock ISSN 1230-3429.
\newblock \doi{10.12775/TMNA.2014.055}.
\newblock URL \url{https://doi.org/10.12775/TMNA.2014.055}.

\bibitem[Amick(1978)]{MR502660}
C.J. Amick.
\newblock Some remarks on {R}ellich's theorem and the {P}oincar\'{e}
  inequality.
\newblock \emph{J. London Math. Soc. (2)}, 18\penalty0 (1):\penalty0 81--93,
  1978.
\newblock ISSN 0024-6107.
\newblock \doi{10.1112/jlms/s2-18.1.81}.
\newblock URL \url{https://doi.org/10.1112/jlms/s2-18.1.81}.

\bibitem[Bennett and Sharpley(1988)]{BS}
C.~Bennett and R.~Sharpley.
\newblock \emph{Interpolation of operators}, volume 129 of \emph{Pure and
  Applied Mathematics}.
\newblock Academic Press, Inc., Boston, MA, 1988.
\newblock ISBN 0-12-088730-4.

\bibitem[Bhattacharya and Leonetti(1991)]{MR1131493}
T.~Bhattacharya and F.~Leonetti.
\newblock A new {P}oincar\'{e} inequality and its application to the regularity
  of minimizers of integral functionals with nonstandard growth.
\newblock \emph{Nonlinear Anal.}, 17\penalty0 (9):\penalty0 833--839, 1991.
\newblock ISSN 0362-546X.
\newblock \doi{10.1016/0362-546X(91)90157-V}.
\newblock URL \url{https://doi.org/10.1016/0362-546X(91)90157-V}.

\bibitem[Br{\'e}zis and Wainger(1980)]{BW}
H.~Br{\'e}zis and S.~Wainger.
\newblock A note on limiting cases of {S}obolev embeddings and convolution
  inequalities.
\newblock \emph{Comm. Partial Differential Equations}, 5\penalty0 (7):\penalty0
  773--789, 1980.
\newblock ISSN 0360-5302.
\newblock \doi{10.1080/03605308008820154}.
\newblock URL \url{https://doi.org/10.1080/03605308008820154}.

\bibitem[Brothers and Ziemer(1988)]{MR929981}
J.E. Brothers and W.P. Ziemer.
\newblock Minimal rearrangements of {S}obolev functions.
\newblock \emph{J. Reine Angew. Math.}, 384:\penalty0 153--179, 1988.
\newblock ISSN 0075-4102.

\bibitem[Brudny\u\i(1979)]{B}
J.A. Brudny\u\i.
\newblock Rational approximation and imbedding theorems.
\newblock \emph{Dokl. Akad. Nauk SSSR}, 247\penalty0 (2):\penalty0 269--272,
  1979.
\newblock ISSN 0002-3264.

\bibitem[Cavaliere and Cianchi(2014)]{MR3237035}
P.~Cavaliere and A.~Cianchi.
\newblock Classical and approximate {T}aylor expansions of weakly
  differentiable functions.
\newblock \emph{Ann. Acad. Sci. Fenn. Math.}, 39\penalty0 (2):\penalty0
  527--544, 2014.
\newblock ISSN 1239-629X.
\newblock \doi{10.5186/aasfm.2014.3933}.
\newblock URL \url{https://doi.org/10.5186/aasfm.2014.3933}.

\bibitem[Chanillo and Wheeden(1985)]{MR805809}
S.~Chanillo and R.L. Wheeden.
\newblock Weighted {P}oincar\'{e} and {S}obolev inequalities and estimates for
  weighted {P}eano maximal functions.
\newblock \emph{Amer. J. Math.}, 107\penalty0 (5):\penalty0 1191--1226, 1985.
\newblock ISSN 0002-9327.
\newblock \doi{10.2307/2374351}.
\newblock URL \url{https://doi.org/10.2307/2374351}.

\bibitem[Cianchi(1996)]{MR1406683}
A.~Cianchi.
\newblock A sharp embedding theorem for {O}rlicz-{S}obolev spaces.
\newblock \emph{Indiana Univ. Math. J.}, 45\penalty0 (1):\penalty0 39--65,
  1996.
\newblock ISSN 0022-2518.
\newblock \doi{10.1512/iumj.1996.45.1958}.
\newblock URL \url{https://doi.org/10.1512/iumj.1996.45.1958}.

\bibitem[Cianchi(2004)]{MR2073127}
A.~Cianchi.
\newblock Optimal {O}rlicz-{S}obolev embeddings.
\newblock \emph{Rev. Mat. Iberoamericana}, 20\penalty0 (2):\penalty0 427--474,
  2004.
\newblock ISSN 0213-2230.
\newblock \doi{10.4171/RMI/396}.
\newblock URL \url{https://doi.org/10.4171/RMI/396}.

\bibitem[Cianchi and Pick(1998)]{CP-ARK}
A.~Cianchi and L.~Pick.
\newblock Sobolev embeddings into {BMO}, {VMO}, and {$L_\infty$}.
\newblock \emph{Ark. Mat.}, 36\penalty0 (2):\penalty0 317--340, 1998.
\newblock ISSN 0004-2080.
\newblock \doi{10.1007/BF02384772}.
\newblock URL \url{https://doi.org/10.1007/BF02384772}.

\bibitem[Cianchi and Pick(2009)]{MR2514054}
A.~Cianchi and L.~Pick.
\newblock Optimal {G}aussian {S}obolev embeddings.
\newblock \emph{J. Funct. Anal.}, 256\penalty0 (11):\penalty0 3588--3642, 2009.
\newblock ISSN 0022-1236.
\newblock \doi{10.1016/j.jfa.2009.03.001}.
\newblock URL \url{https://doi.org/10.1016/j.jfa.2009.03.001}.

\bibitem[Cianchi et~al.(1996)Cianchi, Edmunds, and Gurka]{CEG}
A.~Cianchi, D.E. Edmunds, and P.~Gurka.
\newblock On weighted {P}oincar\'{e} inequalities.
\newblock \emph{Math. Nachr.}, 180:\penalty0 15--41, 1996.
\newblock ISSN 0025-584X.
\newblock \doi{10.1002/mana.3211800103}.
\newblock URL \url{https://doi.org/10.1002/mana.3211800103}.

\bibitem[Cianchi et~al.(2015)Cianchi, Pick, and Slav{\'{\i}}kov{\'a}]{CPS}
A.~Cianchi, L.~Pick, and L.~Slav{\'{\i}}kov{\'a}.
\newblock Higher-order {S}obolev embeddings and isoperimetric inequalities.
\newblock \emph{Adv. Math.}, 273:\penalty0 568--650, 2015.
\newblock ISSN 0001-8708.
\newblock \doi{10.1016/j.aim.2014.12.027}.
\newblock URL \url{https://doi.org/10.1016/j.aim.2014.12.027}.

\bibitem[Cwikel and Pustylnik(2000)]{CP}
M.~Cwikel and E.~Pustylnik.
\newblock Weak type interpolation near ``endpoint'' spaces.
\newblock \emph{J. Funct. Anal.}, 171\penalty0 (2):\penalty0 235--277, 2000.
\newblock ISSN 0022-1236.
\newblock \doi{10.1006/jfan.1999.3502}.
\newblock URL \url{https://doi.org/10.1006/jfan.1999.3502}.

\bibitem[Donaldson(1971)]{DONALDSON1971507}
T.~Donaldson.
\newblock Nonlinear elliptic boundary value problems in {O}rlicz-{S}obolev
  spaces.
\newblock \emph{J. Differential Equations}, 10:\penalty0 507--528, 1971.
\newblock ISSN 0022-0396.
\newblock \doi{10.1016/0022-0396(71)90009-X}.
\newblock URL \url{https://doi.org/10.1016/0022-0396(71)90009-X}.

\bibitem[Edmunds et~al.(2020)Edmunds, Mihula, Musil, and Pick]{EMMP}
D.E. Edmunds, Z.~Mihula, V.~Musil, and L.~Pick.
\newblock Boundedness of classical operators on rearrangement-invariant spaces.
\newblock \emph{J. Funct. Anal.}, 278\penalty0 (4):\penalty0 108341, 2020.
\newblock ISSN 0022-1236.
\newblock \doi{10.1016/j.jfa.2019.108341}.
\newblock URL \url{https://doi.org/10.1016/j.jfa.2019.108341}.

\bibitem[Esposito et~al.(2012)Esposito, Ferone, Kawohl, Nitsch, and
  Trombetti]{MR2989444}
L.~Esposito, V.~Ferone, B.~Kawohl, C.~Nitsch, and C.~Trombetti.
\newblock The longest shortest fence and sharp {P}oincar\'{e}-{S}obolev
  inequalities.
\newblock \emph{Arch. Ration. Mech. Anal.}, 206\penalty0 (3):\penalty0
  821--851, 2012.
\newblock ISSN 0003-9527.
\newblock \doi{10.1007/s00205-012-0545-0}.
\newblock URL \url{https://doi.org/10.1007/s00205-012-0545-0}.

\bibitem[Haj\l{a}sz and Koskela(2000)]{MR1683160}
P.~Haj\l{a}sz and P.~Koskela.
\newblock Sobolev met {P}oincar\'{e}.
\newblock \emph{Mem. Amer. Math. Soc.}, 145\penalty0 (688):\penalty0 x+101,
  2000.
\newblock ISSN 0065-9266.
\newblock \doi{10.1090/memo/0688}.
\newblock URL \url{https://doi.org/10.1090/memo/0688}.

\bibitem[Hansson(1979)]{Ha}
K.~Hansson.
\newblock Imbedding theorems of {S}obolev type in potential theory.
\newblock \emph{Math. Scand.}, 45\penalty0 (1):\penalty0 77--102, 1979.
\newblock ISSN 0025-5521.
\newblock \doi{10.7146/math.scand.a-11827}.
\newblock URL \url{https://doi.org/10.7146/math.scand.a-11827}.

\bibitem[Ho(2016)]{MR3525407}
K.-P. Ho.
\newblock Fourier integrals and {S}obolev imbedding on rearrangement invariant
  quasi-{B}anach function spaces.
\newblock \emph{Ann. Acad. Sci. Fenn. Math.}, 41\penalty0 (2):\penalty0
  897--922, 2016.
\newblock ISSN 1239-629X.
\newblock \doi{10.5186/aasfm.2016.4157}.
\newblock URL \url{https://doi.org/10.5186/aasfm.2016.4157}.

\bibitem[Hurri-Syrj\"{a}nen(1992)]{MR1190332}
R.~Hurri-Syrj\"{a}nen.
\newblock Unbounded {P}oincar\'{e} domains.
\newblock \emph{Ann. Acad. Sci. Fenn. Ser. A I Math.}, 17\penalty0
  (2):\penalty0 409--423, 1992.
\newblock ISSN 0066-1953.
\newblock \doi{10.5186/aasfm.1992.1725}.
\newblock URL \url{https://doi.org/10.5186/aasfm.1992.1725}.

\bibitem[Jaye et~al.(2013)Jaye, Maz'ya, and Verbitsky]{MR2989992}
B.J. Jaye, V.G. Maz'ya, and I.E. Verbitsky.
\newblock Quasilinear elliptic equations and weighted {S}obolev-{P}oincar\'{e}
  inequalities with distributional weights.
\newblock \emph{Adv. Math.}, 232:\penalty0 513--542, 2013.
\newblock ISSN 0001-8708.
\newblock \doi{10.1016/j.aim.2012.09.029}.
\newblock URL \url{https://doi.org/10.1016/j.aim.2012.09.029}.

\bibitem[Kerman and Pick(2006)]{MR2254384}
R.~Kerman and L.~Pick.
\newblock Optimal {S}obolev imbeddings.
\newblock \emph{Forum Math.}, 18\penalty0 (4):\penalty0 535--570, 2006.
\newblock ISSN 0933-7741.
\newblock \doi{10.1515/FORUM.2006.028}.
\newblock URL \url{https://doi.org/10.1515/FORUM.2006.028}.

\bibitem[Kolyada(1989)]{MR1040269}
V.I. Kolyada.
\newblock Rearrangements of functions, and embedding theorems.
\newblock \emph{Uspekhi Mat. Nauk}, 44\penalty0 (5(269)):\penalty0 61--95,
  1989.
\newblock ISSN 0042-1316.
\newblock \doi{10.1070/RM1989v044n05ABEH002287}.
\newblock URL \url{https://doi.org/10.1070/RM1989v044n05ABEH002287}.

\bibitem[Krasnosel'ski\u{\i} and Ruticki\u{\i}(1961)]{MR0126722}
M.A. Krasnosel'ski\u{\i} and Ja.B. Ruticki\u{\i}.
\newblock \emph{Convex functions and {O}rlicz spaces}.
\newblock Translated from the first Russian edition by Leo F. Boron. P.
  Noordhoff Ltd., Groningen, 1961.

\bibitem[Mal\'{y} and Ziemer(1997)]{MR1461542}
J.~Mal\'{y} and W.P. Ziemer.
\newblock \emph{Fine regularity of solutions of elliptic partial differential
  equations}, volume~51 of \emph{Mathematical Surveys and Monographs}.
\newblock American Mathematical Society, Providence, RI, 1997.
\newblock ISBN 0-8218-0335-2.
\newblock \doi{10.1090/surv/051}.
\newblock URL \url{https://doi.org/10.1090/surv/051}.

\bibitem[Maz'ya(2011)]{Mabook}
V.G. Maz'ya.
\newblock \emph{Sobolev spaces with applications to elliptic partial
  differential equations}, volume 342 of \emph{Grundlehren der Mathematischen
  Wissenschaften [Fundamental Principles of Mathematical Sciences]}.
\newblock Springer, Heidelberg, augmented edition, 2011.
\newblock ISBN 978-3-642-15563-5.
\newblock \doi{10.1007/978-3-642-15564-2}.
\newblock URL \url{https://doi.org/10.1007/978-3-642-15564-2}.

\bibitem[Mihula(2019)]{M19}
Z.~Mihula.
\newblock Embeddings of homogeneous {S}obolev spaces on the entire space.
\newblock Preprint. arXiv:1908.03384 [math.FA], 2019.

\bibitem[Musil(2018)]{VejtekPhD}
V.~Musil.
\newblock \emph{Classical operators of harmonic analysis in {O}rlicz spaces}.
\newblock PhD thesis, {C}harles {U}niversity, Faculty of {M}athematics and
  {P}hysics, 2018.
\newblock URL \url{https://is.cuni.cz/webapps/zzp/detail/150069/}.

\bibitem[O'Neil(1963)]{MR0146673}
R.~O'Neil.
\newblock Convolution operators and {$L(p,\,q)$} spaces.
\newblock \emph{Duke Math. J.}, 30:\penalty0 129--142, 1963.
\newblock ISSN 0012-7094.
\newblock \doi{10.1215/S0012-7094-63-03015-1}.
\newblock URL \url{http://projecteuclid.org/euclid.dmj/1077374532}.

\bibitem[Opic and Pick(1999)]{OP}
B.~Opic and L.~Pick.
\newblock On generalized {L}orentz-{Z}ygmund spaces.
\newblock \emph{Math. Inequal. Appl.}, 2\penalty0 (3):\penalty0 391--467, 1999.
\newblock ISSN 1331-4343.
\newblock \doi{10.7153/mia-02-35}.
\newblock URL \url{https://doi.org/10.7153/mia-02-35}.

\bibitem[Payne and Weinberger(1960)]{MR0117419}
L.E. Payne and H.F. Weinberger.
\newblock An optimal {P}oincar\'{e} inequality for convex domains.
\newblock \emph{Arch. Rational Mech. Anal.}, 5:\penalty0 286--292, 1960.
\newblock ISSN 0003-9527.
\newblock \doi{10.1007/BF00252910}.
\newblock URL \url{https://doi.org/10.1007/BF00252910}.

\bibitem[Peetre(1966)]{MR0221282}
J.~Peetre.
\newblock Espaces d'interpolation et th\'{e}or\`eme de {S}oboleff.
\newblock \emph{Ann. Inst. Fourier (Grenoble)}, 16\penalty0 (1):\penalty0
  279--317, 1966.
\newblock ISSN 0373-0956.
\newblock URL \url{http://www.numdam.org/item?id=AIF_1966__16_1_279_0}.

\bibitem[Rabier(2018)]{MR3893783}
P.J. Rabier.
\newblock Uniformly local spaces and refinements of the classical {S}obolev
  embedding theorems.
\newblock \emph{Ark. Mat.}, 56\penalty0 (2):\penalty0 409--440, 2018.
\newblock ISSN 0004-2080.
\newblock \doi{10.4310/ARKIV.2018.v56.n2.a13}.
\newblock URL \url{https://doi.org/10.4310/ARKIV.2018.v56.n2.a13}.

\bibitem[Rao and Ren(1991)]{MR1113700}
M.M. Rao and Z.D. Ren.
\newblock \emph{Theory of {O}rlicz spaces}, volume 146 of \emph{Monographs and
  Textbooks in Pure and Applied Mathematics}.
\newblock Marcel Dekker, Inc., New York, 1991.
\newblock ISBN 0-8247-8478-2.

\bibitem[Stein(1981)]{MR607898}
E.M. Stein.
\newblock Editor's note: the differentiability of functions in {${\bf R}^{n}$}.
\newblock \emph{Ann. of Math. (2)}, 113\penalty0 (2):\penalty0 383--385, 1981.
\newblock ISSN 0003-486X.
\newblock URL \url{http://www.jstor.org/stable/2006989}.

\bibitem[Talenti(1994)]{MR1322313}
G.~Talenti.
\newblock Inequalities in rearrangement invariant function spaces.
\newblock In \emph{Nonlinear analysis, function spaces and applications, {V}ol.
  5 ({P}rague, 1994)}, pages 177--230. Prometheus, Prague, 1994.

\bibitem[Tartar(1998)]{MR1662313}
L.~Tartar.
\newblock Imbedding theorems of {S}obolev spaces into {L}orentz spaces.
\newblock \emph{Boll. Unione Mat. Ital. Sez. B Artic. Ric. Mat. (8)},
  1\penalty0 (3):\penalty0 479--500, 1998.
\newblock ISSN 0392-4041.

\bibitem[Vuillermot(1982)]{VUILLERMOT1982327}
P-A. Vuillermot.
\newblock A class of {O}rlicz-{S}obolev spaces with applications to variational
  problems involving nonlinear {H}ill's equations.
\newblock \emph{J. Math. Anal. Appl.}, 89\penalty0 (1):\penalty0 327--349,
  1982.
\newblock ISSN 0022-247X.
\newblock \doi{10.1016/0022-247X(82)90105-6}.
\newblock URL \url{https://doi.org/10.1016/0022-247X(82)90105-6}.

\bibitem[Zhang(2011)]{MR2676347}
Q.S. Zhang.
\newblock \emph{Sobolev inequalities, heat kernels under {R}icci flow, and the
  {P}oincar\'{e} conjecture}.
\newblock CRC Press, Boca Raton, FL, 2011.
\newblock ISBN 978-1-4398-3459-6.

\bibitem[Ziemer(1989)]{ziemerweakly}
W.P. Ziemer.
\newblock \emph{Weakly differentiable functions}, volume 120 of \emph{Graduate
  Texts in Mathematics}.
\newblock Springer-Verlag, New York, 1989.
\newblock ISBN 0-387-97017-7.
\newblock \doi{10.1007/978-1-4612-1015-3}.
\newblock URL \url{https://doi.org/10.1007/978-1-4612-1015-3}.
\newblock Sobolev spaces and functions of bounded variation.

\bibitem[Zygmund(1959)]{MR0107776}
A.~Zygmund.
\newblock \emph{Trigonometric series. 2nd ed. {V}ols. {I}, {II}}.
\newblock Cambridge University Press, New York, 1959.

\end{thebibliography}
%\nocite{*}

\end{document}